\documentclass[12pt]{amsart}
\usepackage{amssymb,amsmath,amsfonts,latexsym}
\usepackage{bm}
\usepackage[all,cmtip]{xy}
\usepackage{amscd}

\newcommand{\newsection}[1]{\setcounter{equation}{0} \section{#1}}
\newcommand{\bea}{\begin{eqnarray}}
\newcommand{\eea}{\end{eqnarray}}

\newcommand{\clb}{\mathcal{B}}

\newcommand{\cld}{\mathcal{D}}
\newcommand{\cle}{\mathcal{E}}
\newcommand{\clf}{\mathcal{F}}

\newcommand{\clh}{\mathcal{H}}
\newcommand{\clk}{\mathcal{K}}
\newcommand{\cll}{\mathcal{L}}

\newcommand{\cls}{\mathcal{S}}

\newcommand{\clw}{\mathcal{W}}

\newcommand{\z}{\bm{z}}
\newcommand{\w}{\bm{w}}

\newcommand{\D}{\mathbb{D}}

\newcommand{\raro}{\rightarrow}

\def \qed {\hfill \vrule height6pt width 6pt depth 0pt}
\def\textmatrix#1&#2\\#3&#4\\{\bigl({#1 \atop #3}\ {#2 \atop #4}\bigr)}
\def\dispmatrix#1&#2\\#3&#4\\{\left({#1 \atop #3}\ {#2 \atop #4}\right)}
\newcommand{\be}{\begin{equation}}
\newcommand{\ee}{\end{equation}}
\newcommand{\ben}{\begin{eqnarray*}}
\newcommand{\een}{\end{eqnarray*}}

\newcommand{\NI}{\noindent}

\newcommand{\bi}{\begin{itemize}}
\newcommand{\ei}{\end{itemize}}

\theoremstyle{definition}

\theoremstyle{plain}

\newtheorem{thm}{Theorem}[section]
\newtheorem{cor}[thm]{Corollary}

\newtheorem{prop}[thm]{Proposition}
\theoremstyle{definition}

\newtheorem{rem}[thm]{Remark}

\numberwithin{equation}{section}
\let\phi=\varphi

\begin{document}

\title[Invariant subspaces in the polydisc ]{Characterization of Invariant subspaces in the
polydisc}

\author[Maji]{Amit Maji}
\address{Indian Statistical Institute, Statistics and Mathematics Unit, 8th Mile, Mysore Road, Bangalore, 560059, India}
\email{amit.iitm07@gmail.com}

\author[Mundayadan]{Aneesh Mundayadan}
\address{Indian Statistical Institute, Statistics and Mathematics Unit, 8th Mile, Mysore Road, Bangalore, 560059, India}
\email{aneeshkolappa@gmail.com}

\author[Sarkar]{Jaydeb Sarkar}
\address{Indian Statistical Institute, Statistics and Mathematics Unit, 8th Mile, Mysore Road, Bangalore, 560059, India}
\email{jay@isibang.ac.in, jaydeb@gmail.com}

\author[Sankar]{Sankar T. R.}
\address{Indian Statistical Institute, Statistics and Mathematics Unit, 8th Mile, Mysore Road, Bangalore, 560059, India}
\email{sankartr90@gmail.com}



%\today

\begin{abstract}
We give a complete characterization of invariant subspaces for
$(M_{z_1}, \ldots, M_{z_n})$ on the Hardy space $H^2(\mathbb{D}^n)$
over the unit polydisc $\mathbb{D}^n$ in $\mathbb{C}^n$, $n >1$. In
particular, this yields a complete set of unitary invariants for
invariant subspaces for $(M_{z_1}, \ldots, M_{z_n})$ on
$H^2(\mathbb{D}^n)$, $n > 1$. As a consequence, we classify a large
class of $n$-tuples, $n > 1$, of commuting isometries. All of our
results hold for vector-valued Hardy spaces over $\mathbb{D}^n$, $n
> 1$. Our invariant subspace theorem solves the well-known open problem on
characterizations of invariant subspaces of the Hardy space over the
unit polydisc.
\end{abstract}

\maketitle

\section{Introduction}

An important open problem in multivariable operator theory and
function theory of several complex variables is the question of a
Beurling type representations of joint invariant subspaces for the
$n$-tuple of multiplication operators $(M_{z_1}, \ldots, M_{z_n})$
on $H^2(\D^n)$, $n>1$. Here $H^2(\D^n)$ denotes the Hardy space over
the unit polydisc $\D^n$ in $\mathbb{C}^n$ (see Section 2 for
notation and definitions). The main obstacle here seems to be the
subtleties of the theory of holomorphic functions in several complex
variables. This problem is compounded by another difficulty
associated with the complex (and mostly unknown) structure of
$n$-tuples, $n >1$, of commuting isometries on Hilbert spaces.

In this paper, we answer the above question by providing a complete
list of natural conditions on closed subspaces of $H^2(\D^n)$. Here
we use the analytic representations of shift invariant subspaces,
representations of Toeplitz operators on the unit disc, geometry of
tensor product of Hilbert spaces and identification of bounded
linear operators under unitary equivalence to overcome such
difficulties.

As motivation, recall that if $n = 1$, then the celebrated Beurling
theorem \cite{B} says that a non-zero closed subspace $\cls$ of
$H^2(\D)$ is invariant for $M_z$ if and only if there exists an
%an isometry $X : H^2(\D) \raro
%H^2(\D)$ such that $X M_z = M_z X$ and $\cls = X H^2(\D)$. In this
%case it follows that $X = M_{\theta}$ for some
inner function $\theta \in H^\infty(\D)$ such that
\[
\cls = \theta H^2(\D).
\]
Note also that it follows (or the other way around) in particular
from the above representation of $\cls$ that
\[
\cls \ominus z \cls = \theta \mathbb{C},
\]
and so
\[
\cls = \mathop{\oplus}_{m=0}^\infty z^m (\cls \ominus z \cls).
\]
One may now ask whether an analogous characterization holds for
invariant subspaces for $(M_{z_1}, \ldots, M_{z_n})$ on $H^2(\D^n)$,
$n > 1$. However, Rudin's pathological examples (see Rudin
\cite{Ru}, page 70) indicates that the above Beurling type
properties does not hold in general for invariant subspaces for
$(M_{z_1}, \ldots, M_{z_n})$ on $H^2(\D^n)$, $n > 1$: There exist
invariant subspaces $\cls_1$ and $\cls_2$ for $(M_{z_1}, M_{z_2})$
on $H^2(\D^2)$ such that

(1) $\cls_1$ is not finitely generated, and

(2) $\cls_2 \cap H^\infty(\mathbb{D}^2) = \{0\}$.

\NI In fact, Beurling type invariant subspaces for $(M_{z_1},
\ldots, M_{z_n})$ on $H^2(\D^n)$, $n > 1$, are rare. They are
closely connected with the tensor product structure of the Hardy
space (or the product domain $\D^n$).

Therefore, the structure of invariant subspaces for $(M_{z_1},
\ldots, M_{z_n})$ on $H^2(\D^n)$, $n > 1$, is quite complicated. The
list of important works in this area include the papers by Agrawal,
Clark, and Douglas \cite{ACD}, Ahern and Clark \cite{AC}, Douglas
and Yan \cite{DY}, Douglas, Paulsen, Sah and Yan \cite{DPSY}, Guo
\cite{KG}, Fang \cite{F1}, Guo and Yang \cite{GY}, Izuchi \cite{I},
Mandrekar \cite{M}, Putinar \cite{P}, Yang \cite{QY} etc. (also see
the references therein). We also refer to Aleman \cite{AA}, Axler
and Bourdon \cite{AB}, Gamelin \cite{G}, Guo \cite{KG1}, Guo, Sun,
Zheng and Zhong \cite{GZ} and Rudin \cite{WR} for related work on
invariant subspaces in both one and several complex variables.

In this paper, first, we represent $H^2(\D^{n+1})$, $n \geq 1$, by
the $H^2(\D^n)$-valued Hardy space $H^2_{H^2(\D^n)}(\D)$. Under this
identification, we prove that $(M_{z_1}, M_{z_2}, \ldots,
M_{z_{n+1}})$ on $H^2(\D^{n+1})$ corresponds to $(M_z, M_{\kappa_1},
\ldots, M_{\kappa_n})$ on $H^2_{H^2(\D^n)}(\D)$, where $\kappa_i \in
H^\infty_{\clb(H^2(\D^n))}(\D)$, $i = 1, \ldots, n$, is a constant
as well as simple and explicit $\clb(H^2(\D^n))$-valued analytic
function (see Theorem \ref{thm-1}, or part (i) of Theorem
\ref{thm-main} below). Then we prove that a closed subspace $\cls
\subseteq H^2_{H^2(\D^n)}(\D)$ is invariant for $(M_z, M_{\kappa_1},
\ldots, M_{\kappa_n})$ if and only if $\cls$ is of Beurling
\cite{B}, Lax \cite{L} and Halmos \cite{H} type and the
corresponding Beurling, Lax and Halmos inner function solves, in an
appropriate sense, $n$ operator equations explicitly and uniquely
(see Theorem \ref{thm-2}, or part (ii) of Theorem \ref{thm-main}
below, and Theorem \ref{thm-unique}).

Recall that two $m$-tuples, $m \geq 1$, of commuting operators
$(A_1, \ldots, A_m)$ on $\clh$ and $(B_1, \ldots, B_m)$ on $\clk$
are said to be \textit{unitarily equivalent} if there exists a
unitary operator $U : \clh \raro \clk$ such that $U A_i = B_i U$ for
all $i = 1, \ldots, m$.

We now summarize the main contents, namely, Theorems \ref{thm-1} and
\ref{thm-2} restricted to the scalar-valued Hardy space case, of
this paper in the following statement.

\begin{thm}\label{thm-main}
Let $n$ be a natural number, and let $H_n = H^2(\D^n)$. Let
$\kappa_i \in H^\infty_{\clb(H_n)}(\D)$ denote the
$\clb(H_n)$-valued constant function on $\D$ defined by
\[
\kappa_i(w) = M_{z_i} \in \clb(H_n),
\]
for all $w \in \D$, and let $M_{\kappa_i}$ denote the multiplication
operator on $H^2_{H_n}(\D)$ defined by
\[
M_{\kappa_i} f = \kappa_i f,
\]
for all $f \in H^2_{H_n}(\D)$ and $i = 1, \ldots, n$. Then the
following statements hold true:

(i) $(M_{z_1}, M_{z_2} \ldots, M_{z_{n+1}})$ on $H^2(\D^{n+1})$ and
$(M_z, M_{\kappa_1}, \ldots, M_{\kappa_n})$ on $H^2_{H_n}(\D)$ are
unitarily equivalent.

(ii) Let $\cls$ be a closed subspace of $H^2_{H_n}(\D)$. Let $\clw =
\cls \ominus z \cls$, and let
\[
\Phi_i(w) = P_{\clw}(I_{\cls} - w P_{\cls} M_z^*)^{-1}
M_{\kappa_i}|_{\clw},
\]
for all $w \in \D$ and $i = 1, \ldots, n$. Then $\cls$ is invariant
for $(M_z, M_{\kappa_1}, \ldots, M_{\kappa_n})$ if and only if
$(M_{\Phi_1}, \ldots,M_{\Phi_n})$ is an $n$-tuple of commuting
shifts on $H^2_{\clw}(\D)$ and there exist an inner function $\Theta
\in H^\infty_{\clb(\clw, H_n)}(\D)$ such that
\[
\cls = \Theta H^2_{\clw}(\D),
\]
and
\[
\kappa_i \Theta = \Theta \Phi_i,
\]
for all $i = 1, \ldots, n$.
\end{thm}

The representation of the invariant subspace $\cls$, in terms of
$\clw$, $\Theta$ and $\{M_{\Phi_i}\}_{i=1}^n$, in part (ii) above is
unique in an appropriate sense (see Theorem \ref{thm-unique}).
Furthermore, the multiplier $\Phi_i$ can be represented as
\[
\Phi_i(w) = P_{\clw} M_{\Theta} (I_{H^2_{\clw}(\D)} - w M_z^*)^{-1}
M_{\Theta}^* M_{\kappa_i}|_{\clw},
\]
for all $w \in \D$ and $i = 1, \ldots, n$. For a more detailed
discussion on the analytic functions $\{\Phi_i\}_{i=1}^n$ on $\D$ we
refer to Remarks \ref{rem-phi} and \ref{rem-thetaw}.

As an immediate application of the above theorem we obtain the
following (see Corollary \ref{cor-ueqv}): If $\cls \subseteq
H^2_{H_n}(\D)$ is a closed invariant subspace for $(M_z,
M_{\kappa_1}, \ldots, M_{\kappa_n})$, then the tuples $(M_z|_{\cls},
M_{\kappa_1}|_{\cls}, \ldots, M_{\kappa_n}|_{\cls})$ on $\cls$ and
$(M_z, M_{\Phi_1}, \ldots, M_{\Phi_n})$ on $H^2_{\clw}(\D)$ are
unitarily equivalent, where $\clw = \cls \ominus z \cls$ and
\[
\Phi_i(w) = P_{\clw}(I_{\cls} - w P_{\cls} M_z^*)^{-1}
M_{\kappa_i}|_{\clw},
\]
for all $w \in \D$ and $i = 1, \ldots, n$. Our approach also yields
a complete set of unitary invariants for invariant subspaces: The
$n$-tuples of commuting shifts $(M_{\Phi_1}, \ldots, M_{\Phi_n})$ on
$H^2_{\clw}(\D)$ is a complete set of unitary invariants for
invariant subspaces for $(M_z, M_{\kappa_1}, \ldots, M_{\kappa_n})$
on $H^2_{H_n}(\D)$ (see Theorem \ref{th-completeset} for more
details).

We also contribute to the classification problem of commuting tuples
of isometries on Hilbert spaces. It is well known that the structure
of $n$-tuples, $n > 1$, of commuting isometries on Hilbert spaces is
complicated. In fact, except for the Berger, Coburn and Lebow pairs
of commuting isometries \cite{BCL} and Agler and McCarthy pairs of
commuting isometries \cite{AM}, very little is known about the
detailed structure of $n$-tuples of commuting isometries. On the
other hand, $n$-tuples of commuting isometries play a central role
in multivariable operator theory and function theory. In Corollary
\ref{cor-ueqv}, as a byproduct of our analysis, we completely
classify $n$-tuples of commuting isometries of the form
$(M_z|_{\cls}, M_{\kappa_1}|_{\cls}, \ldots, M_{\kappa_n}|_{\cls})$
on $\cls$, where $\cls$ is a closed invariant subspace for $(M_z,
M_{\kappa_1}, \ldots, M_{\kappa_n})$ on $H^2_{\cle_n}(\D)$.

This paper is organized as follows. In Section 2 we give various
background definitions and results on the Hardy space over the unit
polydisc. In Section 3, we prove the central result of this paper -
representations of invariant subspaces of vector-valued Hardy spaces
over polydisc. In chapter 4 we study and analyze the model tuples of
commuting isometries. Section 5 complements the main results on
representations of invariant subspaces and deals with the uniqueness
part. In Section 6 we give some applications related to the main
theorems. The final section of this paper is devoted to an appendix
on a dimension inequality which is relevant to the present context
and of independent interest.

\newsection{Prerequisites}

We start by briefly recalling the relevant parts of the Hardy space
over the unit polydisc. Let $n \geq 1$, and let $\mathbb{D}^n$ be
the open unit polydisc in $\mathbb{C}^n$. The \textit{Hardy space}
$H^2(\mathbb{D}^n)$ over $\mathbb{D}^n$ is the Hilbert space of all
holomorphic functions $f$ on $\mathbb{D}^n$ such that
\[
\|f\|_{H^2(\mathbb{D}^n)} = \left(\sup_{0\leq r< 1}
\int_{\mathbb{T}^n}|f(r e^{i \theta_1}, \ldots, r e^{i
\theta_n})|^2~d {\theta} \right)^{\frac{1}{2}}< \infty,
\]
where $d {\theta}$ is the normalized Lebesgue measure on the torus
$\mathbb{T}^n$, the distinguished boundary of $\mathbb{D}^n$. It is
well known that $H^2(\D^n)$ is a reproducing kernel Hilbert space
corresponding to the Szeg\"{o} kernel $\mathbb{S}_n$ on $\D^n$,
where
\[
\mathbb{S}_n(\z, \w) = \prod_{i=1}^n (1 - z_i \bar{w}_i)^{-1} \quad
\quad (\z, \w \in \D^n).
\]
Clearly
\[
\mathbb{S}_n^{-1}(\z, \w) = \mathop{\sum}_{0 \leq |\bm{k}| \leq n}
(-1)^{|\bm{k}|} \z^{\bm{k}} \bar{\w}^{\bm{k}},
\]
where $|\bm{k}| = \mathop{\sum}_{i=1}^n k_i$ and $0 \leq k_i \leq 1$
for all $i = 1, \ldots, n$. Here we use the notation $\z$ for the
$n$-tuple $(z_1, \ldots, z_n)$ in $\mathbb{C}^n$. Also for any
multi-index $\bm{k} = (k_1, \ldots, k_n) \in \mathbb{Z}_+^n$ and $\z
\in \mathbb{C}^n$, we write $\z^{\bm{k}} = z_1^{k_1} \cdots
z_n^{k_n}$.

Let $\cle$ be a Hilbert space, and let $H^2_{\cle}(\D^n)$ denote the
$\cle$-valued Hardy space over $\D^n$. Then $H^2_{\cle}(\D^n)$ is
the $\cle$-valued reproducing kernel Hilbert space with the
$\clb(\cle)$-valued kernel function
\[
(\z, \w) \mapsto \mathbb{S}_n(\z, \w) I_{\cle} \quad \quad (\z, \w
\in \D^n).
\]
In the sequel, by virtue of the canonical unitary $U :
H^2_{\cle}(\D^n) \raro H^2(\D^n) \otimes \cle$ defined by
\[
U (\z^{\bm{k}} \eta) = \z^{\bm{k}} \otimes \eta \quad \quad (\bm{k}
\in \mathbb{Z}^n_+, \eta \in \cle),
\]
we shall often identify the vector valued Hardy space $H^2(\D^n)
\otimes \cle$ with $H^2_{\cle}(\D^n)$. Let $(M_{z_1}, \ldots,
M_{z_n})$ denote the $n$-tuple of multiplication operators on
$H^2_{\cle}(\D^n)$ by the coordinate functions $\{z_i\}_{i=1}^n$,
that is,
\[
(M_{z_i} f)(\w) = w_i f(\w),
\]
for all $f \in H^2_{\cle}(\D^n)$, $\w \in \D^n$ and $i = 1, \ldots,
n$. It is well known and easy to check that
\[
\|M_{z_i} f\| = \|f\|,
\]
and
\[
\|M_{z_i}^{*m} f\| \raro 0,
\]
as $m \raro \infty$ and for all $f \in H^2_{\cle}(\D^n)$, that is,
$M_{z_i}$ defines a shift (see the definition of shift below) on
$H^2_{\cle}(\D^n)$, $i = 1, \ldots, n$. If $n
> 1$, then it also follows easily that
\[
M_{z_i} M_{z_j} = M_{z_j} M_{z_i},
\]
and
\[
M_{z_i}^* M_{z_j} = M_{z_j} M_{z_i}^*,
\]
for all $1 \leq i < j \leq n$. Therefore, $(M_{z_1}, \ldots,
M_{z_n})$ is an $n$-tuple of \textit{doubly commuting} shifts on
$H^2_{\cle}(\D^n)$. Evidently the shift $M_{z_i}$ on
$H^2_{\cle}(\D^n)$ can be identified with $M_{z_i} \otimes I_{\cle}$
on $H^2(\D^n) \otimes \cle$. This canonical identification will be
used throughout the rest of the paper.

\NI We recall that a closed subspace $\cls \subseteq
H^2_{\cle}(\D^n)$ is called an \textit{invariant subspace} for
$(M_{z_1}, \ldots, M_{z_n})$ on $H^2_{\cle}(\D^n)$ if
\[
z_i \cls \subseteq \cls,
\]
for all $i = 1, \ldots, n$.

Now we review and adapt some standard techniques for shift operators
which are useful for our purposes (see \cite{MSS} for more details).
Let $\clh$ be a Hilbert space. Let $V$ be an isometry on $\clh$,
that is, $\|V f\| = \|f\|$ for all $f \in \clh$. Then $V$ is said to
be a \textit{shift} \cite{H} if there is no non trivial reducing
subspace of $\clh$ on which $V$ is unitary. Equivalently, an
isometry $V$ on $\clh$ is a shift if $V$ is pure, that is, $\|V^{*m}
f\| \raro 0$ for all $f \in \clh$. Now if $V$ is a shift on $\clh$,
then
\[
\clh = \mathop{\oplus}_{m=0}^\infty V^m \clw,
\]
where $\clw$ is the \textit{wandering subspace} \cite{H} for $V$,
that is, $\clw = \ker V^* = \clh \ominus V \clh$. Hence the natural
map $\Pi_V : \clh \raro H^2_{\clw}(\D)$ defined by
\[
\Pi_V (V^m \eta) = z^m \eta,
\]
for all $m \geq 0$ and $\eta\in \clw$, is a unitary operator and
\[
\Pi_V V = M_z \Pi_V.
\]
Following Wold \cite{W} and von Neumann \cite{VN}, we call $\Pi_V$
the \textit{Wold-von Neumann decomposition} of the shift $V$.

We will need the following representation theorem for commutators of
shifts proved in \cite{MSS}. Here we only sketch this proof and
refer the reader to \cite{MSS} for more details.

\begin{thm}\label{thm-commutator}
Let $\clh$ be a Hilbert space. Let $V$ be a shift on $\clh$ and $C$
be a bounded operator on $\clh$. Let $\Pi_V$ be the Wold-von Neumann
decomposition of $V$, $M = \Pi_V C \Pi^{*}_V$, and let
\[
\Theta(z)= P_{\clw}(I_{\clh} -z V^{*})^{-1}C\mid_{\clw} \quad \quad
(z \in \D).
\]
Then $C V = VC$ if and only if $\Theta \in
H^\infty_{\clb(\clw)}(\D)$ and
\[
M = M_{\Theta}.
\]
\end{thm}

\noindent\textit{ Sketch of proof:} For the necessary part, let $CV
= VC$. Then $M M_z = M_z M$, and so
\[
M = M_{\Theta},
\]
for some (unique) bounded analytic function $\Theta \in
H^{\infty}_{\clb(\clw)}(\D)$ \cite{NF}. Let $z \in \D$ and $\eta \in
\clw$. Since $\Theta(z) \eta  = (M_{\Theta} \eta)(z)$, it follows
that
\[
\begin{split}
\Theta(z) \eta & = (\Pi_V C \Pi_V^* \eta)(z)
\\
& = (\Pi_V C \eta)(z),
\end{split}
\]
as $\Pi_V^* \eta = \eta$. Now a simple computation shows that (cf.
\cite{MSS})
\[
I_{\clh} = \sum_{m=0}^\infty V^m P_{\clw} V^{*m},
\]
in the strong operator topology, from which it follows that
\[
C \eta = \sum_{m=0}^{\infty} V^m P_{\clw} V^{*m} C \eta,
\]
and so
\[
\begin{split}
\Theta(z) \eta & = (\Pi_V (\sum_{m=0}^{\infty} V^m P_{\clw} V^{*m} C
\eta))(z)
\\
& = (\sum_{m=0}^{\infty} M_z^m (P_{\clw} V^{*m} C \eta))(z).
\end{split}
\]
 Using the fact that $P_{\clw} V^{*m} C \eta \in \clw$ for all $m \geq 0$, from here
we get
\[
\begin{split}
\Theta(z) \eta & = \sum_{m=0}^{\infty} z^m (P_{\clw} V^{*m} C \eta)
\\
& = P_{\clw} (I_{\clh} - z V^*)^{-1} C \eta.
\end{split}
\]
The sufficient part easily follows from the fact that $\Pi_V^* M
\Pi_V = C$. This proves the theorem. \qed

As usual, here $H^\infty_{\clb(\clw)}(\D)$ denotes the Banach
algebra of all $\clb(\cle)$-valued bounded analytic functions on the
open unit disc $\D$ (cf. \cite{NF}).

\newsection{Main results}

With the above preparation, we now turn to the representations of
joint invariant subspaces of vector-valued Hardy spaces. Let $n$ be
a positive integer. Let $\cle$ be a Hilbert space, and consider the
vector-valued Hardy space $H^2_{\cle}(\D^{n+1})$. Our strategy here
is to identify $M_{z_1}$ on $H^2_{\cle}(\D^{n+1})$ with the
multiplication operator $M_z$ on the $H^2_{\cle}(\D^n)$-valued Hardy
space on the disc $\D$. Then we show that under this identification,
the remaining operators $\{M_{z_2}, \ldots, M_{z_{n+1}}\}$ on
$H^2_{\cle}(\D^{n+1})$ can be represented as the multiplication
operators by $n$ simple and constant $\clb(H^2_{\cle}(\D^n))$-valued
functions on $\D$. For this we need a few more notations.

For each Hilbert space $\cll$, for the sake of notational ease,
define
\[
\cll_n = H^2(\D^n) \otimes \cll.
\]
When $\cll = \mathbb{C}$, we simply write $\cll_n = H_n$, that is,
\[
H_n = H^2(\D^n).
\]
Also, for each $i = 1, \ldots, n$, we define
\[
\kappa_{\cll, i} (w) = M_{z_i} \otimes I_{\cll},
\]
for all $w \in \D$, and write
\[
\kappa_{\cll, i} = \kappa_i,
\]
when $\cll$ is clear from the context. It is evident that $\kappa_i
\in H^\infty_{\clb(\cll_n)}(\D)$ is a constant function and
$M_{\kappa_i}$ on $H^2_{\cll_n}(\D)$, defined by
\[
M_{\kappa_i} f = \kappa_i f \quad \quad (f \in H^2_{\cll_n}(\D)),
\]
is a shift on $H^2_{\cll_n}(\D)$ for all $i = 1, \ldots, n$.

Now we return to the invariant subspaces of $H^2_{\cle}(\D^{n+1})$.
First we identify $H^2_{\cle}(\D^{n+1})$ with $H^2(\D) \otimes
\cle_n$ by the natural unitary map $\hat{U} : H^2_{\cle}(\D^{n+1})
\raro H^2(\D) \otimes \cle_n$ defined by
\[
\hat{U}(z_1^{k_1} z_2^{k_2} \cdots z_{n+1}^{k_{n+1}} \eta) = z^{k_1}
\otimes (z_1^{k_2} \cdots z_{n}^{k_{n+1}} \eta),
\]
for all $k_1, \ldots, k_{n+1} \geq 0$ and $\eta \in \cle$. Then it
is clear that
\[
\hat{U} M_{z_1} = (M_z \otimes I_{\cle_n}) \hat{U}.
\]
Moreover, a simple computation shows that
\[
\hat{U} M_{z_{1 + i}} = (I_{H^2(\D)} \otimes K_i) \hat{U},
\]
where $K_i$ is the multiplicational operator $M_{z_i}$ on $\cle_n$,
that is
\[
K_i = M_{z_i},
\]
for all $i = 1, \ldots, n$. Therefore, the tuples $(M_{z_1},
M_{z_2}, \ldots, M_{z_{n+1}})$ on $H^2_{\cle}(\D^{n+1})$ and $(M_z
\otimes I_{\cle_n}, I_{H^2(\D)} \otimes K_1, \ldots, I_{H^2(\D)}
\otimes K_n)$ on $H^2(\D) \otimes \cle_n$ are unitarily equivalent.
We further identify $H^2(\D) \otimes \cle_n$ with the
$\cle_n$-valued Hardy space $H^2_{\cle_n}(\D)$ by the canonical
unitary map $\tilde{U} : H^2(\D) \otimes \cle_n \raro
H^2_{\cle_n}(\D)$ defined by
\[
\tilde{U}(z^k \otimes \eta) = z^k \eta,
\]
for all $k \geq 0$ and $\eta \in \cle_n$. Clearly
\[
\tilde{U} (M_z \otimes I_{\cle_n}) = M_z \tilde{U}.
\]
Now for each $i = 1, \ldots, n$, define the constant
$\clb(\cle_n)$-valued (analytic) function on $\D$ by
\[
\kappa_i(z) = K_i,
\]
for all $z \in \D$. Then $\kappa_i \in H^\infty_{\clb(\cle_n)}(\D)$,
and the multiplication operator $M_{\kappa_i}$ on
$H^2_{\cle_n}(\D)$, defined by
\[
(M_{\kappa_i} (z^m \eta))(w) = w^m (K_i \eta),
\]
for all $m \geq 0$, $\eta \in \cle_n$ and $w \in \D$, is a shift on
$H^2_{\cle_n}(\D)$. It is now easy to see that
\[
\tilde{U} (I_{H^2(\D)} \otimes K_i) = M_{\kappa_i} \tilde{U}.
\]
for all $i = 1, \ldots, n$. Finally, by setting
\[
U = \tilde{U} \hat{U},
\]
it follows that $U : H^2_{\cle}(\D^{n+1}) \raro H^2_{\cle_n}(\D)$ is
a unitary operator and
\[
U M_{z_1} = M_z U,
\]
and
\[
U M_{z_{1 + i}} = M_{\kappa_i} {U},
\]
for all $i = 1, \ldots, n$. This proves the vector-valued version of
the first half of the statement of Theorem \ref{thm-main}:

\begin{thm}\label{thm-1}
Let $\cle$ be a Hilbert space. Then $(M_{z_1}, M_{z_2} \ldots,
M_{z_{n+1}})$ on $H^2_{\cle}(\D^{n+1})$ and $(M_z, M_{\kappa_1},
\ldots, M_{\kappa_n})$ on $H^2_{\cle_n}(\D)$ are unitarily
equivalent, where $\kappa_i \in H^\infty_{\clb(\cle_n)}(\D)$ is the
constant function
\[
\kappa_i(w) = M_{z_i} \in \clb(\cle_n),
\]
for all $w \in \D$ and $i = 1, \ldots, n$.
\end{thm}

Now we proceed to prove the remaining half of Theorem \ref{thm-main}
in the vector-valued Hardy space setting. Let $\cls \subseteq
H^2_{\cle_n}(\D)$ be a closed invariant subspace for $(M_z,
M_{\kappa_1}, \ldots, M_{\kappa_n})$ on $H^2_{\cle_n}(\D)$. Set
\[
V = M_z|_{\cls},
\]
and
\[
V_i = M_{\kappa_i}|_{\cls},
\]
for all $i = 1, \ldots, n$. Clearly, $(V, V_1, \ldots, V_n)$ is a
commuting tuple of isometries on $\cls$. Note that if $f \in \cls$,
then
\[
\begin{split}
\|V_i^{*m} f\|_{\cls} & = \|P_{\cls} M_{\kappa_i}^{*m} f\||_{\cls}
\\
& \leq \|M_{\kappa_i}^{*m} f\||_{H^2_{\cle_n}(\D)},
\end{split}
\]
that is, $V_i$, $i = 1, \ldots, n$, is a shift on $\cls$, and
similarly $V$ is also a shift on $\cls$. Let $\clw = \cls \ominus V
\cls$ denote the wandering subspace for $V$, that is
\[
\begin{split}
\clw & = \ker V^*
\\
& = \ker P_{\cls} M_z^*,
\end{split}
\]
and let $\Pi_V : \cls \raro H^2_{\clw}(\D)$ be the Wold-von Neumann
decomposition of $V$ on $\cls$ (see Section 2). Then $\Pi_V$ is a
unitary operator and
\[
\Pi_V V = M_z \Pi_V.
\]
Since
\[
V V_i= V_i V,
\]
applying Theorem \ref{thm-commutator} to $V_i$, we obtain
\[
\Pi_V V_i = M_{\Phi_i} \Pi_V,
\]
where
\[
\Phi_i(w) = P_{\clw}(I_{\cls} - w V^*)^{-1} V_i|_{\clw},
\]
for all $w \in \D$, $\Phi_i \in H^\infty_{\clb(\clw)}(\D)$,
$M_{\Phi_i}$ is a shift on $H^2_{\clw}(\D)$ since $V_i$ is a shift
on $\cls$ and $i = 1, \ldots, n$. Now since $\Pi_V$ is unitary, we
obtain that
\[
\Pi_V^* M_z = V \Pi_V^*,
\]
and
\[
\Pi_V^* V_i = M_{\Phi_i} \Pi_V^*,
\]
for all $i = 1, \ldots, n$. Finally, if we let $i_{\cls}$ denote the
inclusion map $i_{\cls} : \cls \hookrightarrow H^2_{\cle_n}(\D)$,
then $\Pi_{\cls} : H^2_{\clw}(\D) \raro H^2_{\cle_n}(\D)$ is an
isometry, where
\[
\Pi_{\cls} = i_{\cls} \circ \Pi_V^*.
\]
Clearly $\Pi_{\cls} \Pi_{\cls}^* = i_{\cls} i_{\cls}^*$. This
implies that
\[
\mbox{ran~} \Pi_{\cls} = \mbox{ran~} i_{\cls},
\]
and so
\[
\mbox{ran~} \Pi_{\cls} =\cls.
\]
Now, using $i_{\cls} V = M_z i_{\cls}$ and $i_{\cls} V_j =
M_{\kappa_j} i_{\cls}$, we have
\[
\Pi_{\cls} M_z = M_z \Pi_{\cls},
\]
and
\[
\Pi_{\cls} M_{\Phi_i} = M_{\kappa_i} \Pi_{\cls},
\]
for all $i = 1, \ldots, n$. From the first equality it follows that
there exists an inner function $\Theta \in H^\infty_{\clb(\clw,
\cle_n)}(\D)$ such that
\[
\Pi_{\cls} = M_{\Theta}.
\]
This and the second equality implies that
\[
\kappa_i \Theta = \Theta \Phi_i,
\]
for all $i = 1, \ldots, n$. Moreover, $\text{ran~} \Pi_{\cls} =
\cls$ yields
\[
\cls = \Theta H^2_{\clw}(\D).
\]
To prove that $(M_{\Phi_1}, \ldots, M_{\Phi_n})$ is a commuting
tuple, observe that
\[
\begin{split}
M_{\Phi_i} M_{\Phi_j} \Pi_V & = M_{\Phi_i} \Pi_V V_j
\\
& = \Pi_V V_i V_j
\\
& = \Pi_V V_j V_i
\\
& = M_{\Phi_j} M_{\Phi_i} \Pi_V,
\end{split}
\]
and so
\[
M_{\Phi_i} M_{\Phi_j} = M_{\Phi_j} M_{\Phi_i},
\]
for all $i, j = 1, \ldots, n$. This proves the last part of Theorem
\ref{thm-main} in the vector-valued Hardy space setting:

\begin{thm}\label{thm-2}
Let $\cle$ be a Hilbert space. Let $\cls \subseteq H^2_{\cle_n}(\D)$
be a closed subspace, $\clw = \cls \ominus z \cls$, and let
\[
\Phi_i(w) = P_{\clw}(I_{\cls} - w P_{\cls} M_z^*)^{-1}
M_{\kappa_i}|_{\clw},
\]
for all $w \in \D$ and $i = 1, \ldots, n$. Then $\cls$ is invariant
for $(M_z, M_{\kappa_1}, \ldots, M_{\kappa_n})$ if and only if
$(M_{\Phi_1}, \ldots,M_{\Phi_n})$ is an $n$-tuple of commuting
shifts on $H^2_{\clw}(\D)$ and there exists an inner function
$\Theta \in H^\infty_{\clb(\clw, \cle_n)}(\D)$ such that
\[
\cls = \Theta H^2_{\clw}(\D),
\]
and
\[
\kappa_i \Theta = \Theta \Phi_i,
\]
for all $i = 1, \ldots, n$.
\end{thm}

A few remarks are in order.

\begin{rem}\label{rem-phi}
Note that since $\|w P_{\cls} M_z^*\| < 1$ for all $w \in \D$, the
$\clb(\clw)$-valued function $\Phi_i$, as defined in the above
theorem, is analytic on $\D$. Here the boundedness condition (or the
shift condition) on $M_{\Phi_i}$ on $H^2_{\clw}(\D)$ assures that
$\Phi_i \in H^\infty_{\clb(\clw)}(\D)$ for all $i = 1, \ldots, n$.
\end{rem}

\begin{rem}\label{rem-Beur}
Clearly, one obvious necessary condition for a closed subspace
$\cls$ of $H^2_{\cle_n}(\D)$ to be invariant for $(M_z,
M_{\kappa_1}, \ldots, M_{\kappa_n})$ is that $\cls$ is invariant for
$M_z$, and, consequently
\[
\cls = \Theta H^2_{\clw}(\D),
\]
is the classical Beurling, Lax and Halmos representation of $\cls$,
where $\clw = \cls \ominus z \cls$ is the wandering subspace for
$M_z|_{\cls}$ and $\Theta \in H^\infty_{\clb(\clw, \cle_n)}(\D)$ is
the (unique up to a unitary constant right factor; see Section 5)
Beurling, Lax and Halmos inner function. Moreover, since $\kappa_i
\cls \subseteq \cls$, another condition which is evidently necessary
(by Douglas's range inclusion theorem) is that
\[
\kappa_i \Theta = \Theta \Gamma_i,
\]
for some $\Gamma_i \in \clb(H^2_{\clw}(\D))$, $i = 1, \ldots, n$. In
the above theorem, we prove that $\Gamma_i$ is explicit, that is
\[
 \Gamma_i = \Phi_i \in H^\infty_{\clb(\clw)}(\D),
\]
for all $i = 1, \ldots, n$, and $(\Gamma_1, \ldots, \Gamma_n)$ is an
$n$-tuple of commuting shifts on $H^2_{\clw}(\D)$. This is probably
the most non-trivial part of our treatment to the invariant subspace
problem in the present setting.
\end{rem}

\begin{rem}\label{rem-thetaw}
Let $\cle$ be a Hilbert space, and let $\cls \subseteq
H^2_{\cle_n}(\D)$ be a closed invariant subspace for $(M_z,
M_{\kappa_1}, \ldots, M_{\kappa_n})$ on $H^2_{\cle_n}(\D)$. Let
$\clw$, $\Theta$ and $\{\Phi_i\}_{i=1}^n \subseteq
H^\infty_{\clb(\clw, \cle_n)}(\D)$ be as in Theorem \ref{thm-2}. Now
it follows from $P_{\cls} = M_{\Theta} M_{\Theta}^*$ that $P_{\cls}
M_z^{*m} = M_{\Theta} M_z^{*m} M_{\Theta}^*$ for all $m \geq 0$.
Hence the equality
\[
(I_{\cls} - w P_{\cls} M_z^*)^{-1} = \sum_{m=0}^\infty w^m P_{\cls}
M_z^{*m},
\]
yields
\[
(I_{\cls} - w P_{\cls} M_z^*)^{-1} = M_{\Theta} (I_{H^2_{\clw}(\D)}
- w M_z^*)^{-1} M_{\Theta}^*,
\]
so that
\[
\Phi_i(w) = P_{\clw} M_{\Theta} (I_{H^2_{\clw}(\D)} - w M_z^*)^{-1}
M_{\Theta}^* M_{\kappa_i}|_{\clw},
\]
for all $w \in \D$ and $i = 1, \ldots, n$.
\end{rem}

A well known consequence of the Beurling, Lax and Halmos theorem
(cf. page 239, Foias and Frazho \cite{FF}) implies that a closed
subspace $\cls \subseteq H^2_{\cle}(\D)$ is invariant for $M_z$ if
and only if $\cls \cong H^2_{\clf}(\D)$ for some Hilbert space
$\clf$ with
\[
\text{dim~} \clf \leq \text{dim~} \cle.
\]
More specifically, if $\cls$ is a closed invariant subspace of
$H^2_{\cle}(\D)$ and if  $\clw = \cls \ominus z \cls$, then the pure
isometry $M_z|_{\cls}$ on $\cls$ and $M_z$ on $H^2_{\clw}(\D)$ are
unitarily equivalent, and $\text{dim~} \clw \leq \text{dim~} \cle$.
The above theorem sets the stage for a similar result.

\begin{cor}\label{cor-ueqv}
Let $\cle$ be a Hilbert space, and let $\cls \subseteq
H^2_{\cle_n}(\D)$ be a closed invariant subspace for $(M_z,
M_{\kappa_1}, \ldots, M_{\kappa_n})$ on $H^2_{\cle_n}(\D)$. Let
$\clw = \cls \ominus z \cls$, and
\[
\Phi_i(w) = P_{\clw}(I_{\cls} - w P_{\cls} M_z^*)^{-1}
M_{\kappa_i}|_{\clw},
\]
for all $w \in \D$ and $i = 1, \ldots, n$. Then $(M_z|_{\cls},
M_{\kappa_1}|_{\cls}, \ldots, M_{\kappa_n}|_{\cls})$ on $\cls$ and
$(M_z, M_{\Phi_1}, \ldots, M_{\Phi_n})$ on $H^2_{\clw}(\D)$ are
unitarily equivalent.
\end{cor}
\begin{proof}
Let $\clw$, $\Theta$ and $\{\Phi_i\}_{i=1}^n \subseteq
H^\infty_{\clb(\clw)}(\D)$ be as in Theorem \ref{thm-2}. Then it
follows that
\[
X : H^2_{\clw}(\D) \raro \Theta H^2_{\clw}(\D) = \cls,
\]
is a unitary operator, where
\[
X = M_{\Theta}.
\]
It is now clear that $X$ intertwines $(M_z, M_{\Phi_1}, \ldots,
M_{\Phi_n})$ on $H^2_{\clw}(\D)$ and $(M_z|_{\cls},
M_{\kappa_1}|_{\cls}, \ldots, M_{\kappa_n}|_{\cls})$ on $\cls$. This
completes the proof of the corollary.
\end{proof}

Let $\cle$ be a Hilbert space, and let $\cls \subseteq
H^2_{\cle_n}(\D)$ be an invariant subspace for $M_z$. Then $\cls =
\Theta H^2_{\clw}(\D)$, where $\clw = \cls \ominus z \cls$ and
$\Theta \in H^\infty_{\clb(\clw, \cle_n)}(\D)$ is the Beurling, Lax
and Halmos inner function. A natural question arises in connection
with Remark \ref{rem-Beur}: Under what additional condition(s) on
$\Theta$ is $\cls$ also invariant for $(M_{\kappa_1}, \ldots,
M_{\kappa_n})$? An answer to this question directly follows, with
appropriate reformulation, from Theorem \ref{thm-2} and Remark
\ref{rem-thetaw}:

\begin{thm}\label{th-unit equiv}
Let $\cle$ be a Hilbert space, and let $\cls \subseteq
H^2_{\cle_n}(\D)$ be an invariant subspace for $M_z$ on
$H^2_{\cle_n}(\D)$. Let $\cls = \Theta H^2_{\clw}(\D)$, where $\clw
= \cls \ominus z \cls$ and $\Theta \in H^\infty_{\clb(\clw,
\cle_n)}(\D)$ is the Beurling Lax and Halmos inner function. Set
\[
\Phi_i(w) = P_{\clw} M_{\Theta} (I_{H^2_{\clw}(\D)} - w M_z^*)^{-1}
M_{\Theta}^* M_{\kappa_i}|_{\clw},
\]
for all $w \in \D$ and $i = 1, \ldots, n$. Then $\cls$ is invariant
for $(M_{\kappa_1}, \ldots, M_{\kappa_n})$ if and only if
$(M_{\Phi_1}, \ldots, M_{\Phi_n})$ on $H^2_{\clw}(\D)$ is an
$n$-tuple of commuting shifts, and
\[
\kappa_i \Theta = \Theta \Phi_i,
\]
for all $i = 1, \ldots, n$. Moreover, in this case, $(M_z|_{\cls},
M_{\kappa_1}|_{\cls}, \ldots, M_{\kappa_n}|_{\cls})$ on $\cls$ and
$(M_z, M_{\Phi_1}, \ldots, M_{\Phi_n})$ on $H^2_{\clw}(\D)$ are
unitarily equivalent.
\end{thm}

Thus the $n$-tuples of commuting shifts $(M_{\Phi_1}, \ldots,
M_{\Phi_n})$ on $H^2_{\cll}(\D)$, for Hilbert spaces $\cll$ and
inner multipliers $\{\Phi_i\}_{i=1}^n \subseteq
H^\infty_{\clb(\cll)}(\D)$, yielding invariant subspaces of
vector-valued Hardy spaces over $\D^{n+1}$ are distinguished among
the general $n$-tuples of commuting shifts by the fact that
\[
\Phi_i(w) = P_{\cll}(I_{\cls} - w P_{\cls} M_z^*)^{-1}
M_{\kappa_i}|_{\cll} \quad \quad (w \in \D),
\]
where $\cls = \Theta H^2_{\cll}(\D)$ for some inner function $\Theta
\in H^\infty_{\clb(\cll, \cle_n)}(\D)$, and
\[
\kappa_i \Theta = \Theta \Phi_i,
\]
for all $i = 1, \ldots, n$. Moreover, in view of Remark
\ref{rem-thetaw}, the above condition is equivalent to the condition
that
\[
\Phi_i(w) = P_{\clw} M_{\Theta} (I_{H^2_{\cll}(\D)} - w M_z^*)^{-1}
M_{\Theta}^* M_{\kappa_i}|_{\clw},
\]
for some inner function $\Theta \in H^\infty_{\clb(\cll,
\cle_n)}(\D)$ such that
\[
\kappa_i \Theta = \Theta \Phi_i,
\]
for all $i = 1, \ldots, n$.

\newsection{Representations of model isometries}\label{sec-nshift}

In connection with Theorem \ref{thm-1} (or part (i) of Theorem
\ref{thm-main}), a natural question arises: Given a Hilbert space
$\cle$, how to identify Hilbert spaces $\clf$ and
$\clb(\clf)$-valued multipliers $\{\Psi\}_{i=1}^n \subseteq
H^\infty_{\clb(\clf)}(\D)$ such that $(M_z, M_{\Psi_1}, \ldots,
M_{\Psi_n})$ on $H^2_{\clf_n}(\D)$ and $(M_{z}, M_{\kappa_1},
\ldots, M_{\kappa_{n}})$ on $H^2_{\cle_n}(\D)$ are unitarily
equivalent. More generally, given a Hilbert space $\cle$,
characterize $(n+1)$-tuples of commuting shifts on Hilbert spaces
that are unitarily equivalent to $(M_{z}, M_{\kappa_1}, \ldots,
M_{\kappa_{n}})$ on $H^2_{\cle_n}(\D)$.

This question has a simple answer, although a rigorous proof of it
involves some technicalities. More specifically, the answer to this
question is related to a numerical invariant, the rank of an
operator associated with the Szeg\"{o} kernel on $\D^{n+1}$. First,
however, we need a few more definitions.

Let $(T_1, \ldots, T_m)$ be an $m$-tuple of commuting contractions
on a Hilbert space $\clh$. Define the \textit{defect operator}
\cite{GY} corresponding to $(T_1, \ldots, T_m)$ as
\[
\mathbb{S}_m^{-1}(T_1, \ldots, T_m) = \mathop{\sum}_{0 \leq |\bm{k}|
\leq m} (-1)^{|\bm{k}|} T_1^{k_1} \cdots T_m^{k_m} T_1^{* k_1}
\cdots T_m^{* k_m},
\]
where $0 \leq k_i \leq 1$, $i = 1, \ldots, m$. This definition is
motivated by the representation of the Szeg\"{o} kernel on the
polydisc $\D^m$ (see Section 2). We say that $(T_1, \ldots, T_m)$ is
of \textit{rank} $p$ ($p \in \mathbb{N} \cup \{\infty\}$) if
\[
\mbox{rank~} [\mathbb{S}_{m}^{-1}(T_1, \ldots, T_m)] = p,
\]
and we write
\[
\text{rank~} (T_1, \ldots, T_m) = p.
\]
The defect operators plays an important role in multivariable
operator theory (cf. \cite{KG, GY}). For instance, if $\cle$ is a
Hilbert space, then the defect operator of the multiplication
operator tuple $(M_{z_1}, \ldots, M_{z_n})$ on $H^2_{\cle}(\D^n)$ is
given by
\[
\mathbb{S}_n^{-1}(M_{z_1}, \ldots, M_{z_n}) = P_{H^2_c(\D^n)}
\otimes I_{\cle},
\]
where $P_{H^2_c(\D^n)}$ denotes the orthogonal projection of
$H^2(\D^n)$ onto the one dimensional space of constant functions.
Furthermore, as is evident from the definition (and also see the
proof of Theorem \ref{thm-1}), the defect operator for $(M_z,
M_{\kappa_1}, \ldots, M_{\kappa_n})$ on $H^2_{\cle_n}(\D)$ is given
by
\[
\mathbb{S}_{n+1}^{-1}(M_z, M_{\kappa_1}, \ldots, M_{\kappa_n}) =
P_{H^2_c(\D)} \otimes P_{H^2_c(\D^n)} \otimes I_{\cle}.
\]
In particular,
\[
\mbox{dim~} \cle = \mbox{rank~} (M_z, M_{\kappa_1}, \ldots,
M_{\kappa_n}) = \text{rank~} (M_{z_1}, \ldots, M_{z_n}).
\]

Now let $\cle$ and $\clk$ be Hilbert spaces, and let $(V, V_1
\ldots, V_{n})$ be an $(n+1)$-tuple of commuting shifts on $\clk$.
Suppose that $(V, V_1 \ldots, V_{n})$ on $\clk$ and $(M_{z},
M_{\kappa_1}, \ldots, M_{\kappa_{n}})$ on $H^2_{\cle_n}(\D)$ are
unitarily equivalent. In this case, it is necessary that $M_z$ on
$H^2_{\cle_n}(\D)$ and $V$ on $\clk$ are unitarily equivalent. As $V
V_i = V_i V$ and $V_i V_j = V_j V_i$ for all $i, j = 1, \ldots, n$,
Theorem \ref{thm-commutator} implies that $(V, V_1, \ldots, V_n)$
and $(M_z, M_{\Phi_1}, \ldots, M_{\Phi_n})$ on $H^2_{\clw}(\D)$ are
unitarily equivalent, where $\clw = \clk \ominus V \clk$, and
\[
\Phi_i(z) = P_{\clw} (I_{\clk} - z V^*)^{-1} V_i|_{\clw},
\]
for all $z \in \D$ and $i = 1, \ldots, n$. Since $(M_{z},
M_{\kappa_1}, \ldots, M_{\kappa_{n}})$ on $H^2_{\cle_n}(\D)$ is
doubly commuting, another necessary condition is that $(V, V_1,
\ldots, V_{n+1})$ is doubly commuting. In particular, $V^* V_i = V_i
V^*$, and so
\[
V^{*m} V_i = V_i V^{*m},
\]
for all $m \geq 0$ and $i = 1, \ldots, n$. Using $V^{*m}|_{\clw} =
0$ for all $m \geq 1$, this implies that $\Phi_i(z) = P_{\clw}
V_i|_{\clw}$ for all $z \in \D$. Again using $V V_i^* = V_i^* V$, we
have
\[
V_i (I - V V^*) = (I - V V^*) V_i,
\]
for all $i = 1, \ldots, n$. This implies that $\clw$ is a reducing
subspace for $V_i$, and hence we obtain
\[
\Phi_i(z) = V_i|_{\clw},
\]
that is, ${\Phi_i}$ is a constant shift-valued function on $\D$ for
all $i = 1, \ldots, n$. This observation leads to the following
proposition:

\begin{prop}\label{prop-1}
Let $(V, V_1, \ldots, V_n)$ be an $(n+1)$-tuple of doubly commuting
shifts on some Hilbert space $\clh$. Let $\clw = \clh \ominus V
\clh$, and let
\[
\Phi_i(z) = V_i|_{\clw} \quad \quad (i = 1, \ldots, n),
\]
for all $z \in \D$. Then $\clw$ is reducing for $V_i$, $i = 1,
\ldots, n$, and $(V, V_1, \ldots, V_n)$ and $(M_z, M_{\Phi_1},
\ldots, M_{\Phi_n})$ on $H^2_{\clw}(\D)$ are unitarily equivalent.
\end{prop}

In particular, if $\cll$ is a Hilbert space and $(M_z, M_{\Phi_1},
\ldots, M_{\Phi_n})$ on $H^2_{\cll}(\D)$, for some
$\{\Phi_i\}_{i=1}^n \subseteq H^\infty_{\clb(\cll)}(\D)$, is a tuple
of doubly commuting shifts, then
\[
\Phi_i(z) = \Phi_i(0) \quad \quad (z \in \D),
\]
that is, $\Phi$ is a constant function for all $i = 1, \ldots, n$.

Now we return to $(V, V_1 \ldots, V_{n})$, which in turn is an
$(n+1)$-tuple of doubly commuting shifts on $\clh$. For simplicity
of notation, set $U_1 = V$, $U_{i+1} = V_i$ for all $i = 1, \ldots,
n$, and let
\[
\cld = \text{ran~} \mathbb{S}_{n+1}^{-1}(V, V_1, \ldots, V_n) =
\mathop{\cap}_{i=1}^{n+1} \ker U_i^*,
\]
is the wandering subspace for $(V, V_1, \ldots, V_{n})$ (cf.
\cite{JS}). From here, one can use the fact that (cf. Theorem 3.3 in
\cite{JS})
\[
\clh = \mathop{\oplus}_{\bm{k} \in \mathbb{Z}^{n+1}_+} U^{\bm{k}}
\cld,
\]
to prove that the map $\Gamma : \clh \raro H^2_{\cld}(\D^{n+1})$
defined by
\[
\Gamma(U^{\bm{k}} \eta) = \z^{\bm{k}} \eta \quad \quad (\bm{k} \in
\mathbb{Z}^{n+1}_+, \eta \in \cld),
\]
is a unitary and
\[
\Gamma U_i = M_{z_i} \Gamma,
\]
for all $i = 1, \ldots, n+1$. Therefore, $(V, V_1, \ldots, V_n)$ on
$\clh$ and $(M_{z_1}, \ldots, M_{z_{n+1}})$ on
$H^2_{\cld}(\D^{n+1})$ are unitarily equivalent. In addition, if
$\cle$ is a Hilbert space, and
\[
\dim \cle = \text{rank~ }(V, V_1, \ldots, V_n) \;\; ( = \dim \cld),
\]
then it follows that (see the equivalence of (ii) and (v) of Theorem
3.3 in \cite{JS}) $(M_{z_1}, \ldots, M_{z_{n+1}})$ on
$H^2_{\cld}(\D^{n+1})$ and $(M_{z_1}, \ldots, M_{z_{n+1}})$ on
$H^2_{\cle}(\D^{n+1})$ are unitarily equivalent. But then Theorem
\ref{thm-1} yields immediately that $(M_{z_1}, \ldots, M_{z_{n+1}})$
on $H^2_{\cld}(\D^{n+1})$ and $(M_{z}, M_{\kappa_1}, \ldots,
M_{\kappa_{n}})$ on $H^2_{\cle_n}(\D)$ are unitarily equivalent.
This gives the following:

\begin{thm}\label{thm-dcequiv}
In the setting of Proposition \ref{prop-1} the following hold: $(V,
V_1, \ldots, V_n)$ on $\clh$, $(M_z, M_{\Psi_1}, \ldots,
M_{\Psi_n})$ on $H^2_{\clw}(\D)$, and $(M_z, M_{\kappa_1}, \ldots,
M_{\kappa_n})$ on $H^2_{\cle_n}(\D)$ are unitarily equivalent, where
$\cle$ is a Hilbert space and
\[
\mbox{dim~} \cle = \mbox{rank~} (V, V_1, \ldots, V_n).
\]
\end{thm}

Therefore, an $(n+1)$-tuple of doubly commuting shift operators
$(M_z, M_{\Phi_1}, \ldots, M_{\Phi_n})$ is completely determined by
the numerical invariant $\mbox{rank~} (M_z, M_{\Phi_1}, \ldots,
M_{\Phi_n})$:

\begin{cor}
Let $\cle$ and $\clf$ be Hilbert spaces. Let $(M_z, M_{\Psi_1},
\ldots, M_{\Psi_n})$ be an $(n+1)$-tuple of commuting shifts on
$H^2_{\clf}(\D)$. Then $(M_z, M_{\Psi_1}, \ldots, M_{\Psi_n})$ on
$H^2_{\clf}(\D)$ and $(M_z, M_{\kappa_1}, \ldots, M_{\kappa_n})$ on
$H^2_{\cle_n}(\D)$ are unitarily equivalent if and only if $(M_z,
M_{\Psi_1}, \ldots, M_{\Psi_n})$ is doubly commuting and
\[
\mbox{dim~} \cle = \mbox{rank~} (M_z, M_{\Psi_1}, \ldots,
M_{\Psi_n}).
\]
\end{cor}

The above corollary should be compared with the uniqueness of the
multiplicity of shift operators on Hilbert spaces \cite{H}.

\newsection{Nested invariant subspaces and uniqueness}

Now we proceed to the description of nested invariant subspaces of
$H^2_{\cle_n}(\D)$. Let $\cls_1$ and $\cls_2$ be two closed
invariant subspaces for $(M_z, M_{\kappa_1}, \ldots, M_{\kappa_n})$
on $H^2_{\cle_n}(\D)$. Let $\clw_j = \cls \ominus z \cls_j$, and let
\[
\Phi_{j,i}(w) = P_{\clw_j}(I_{\cls_j} - w P_{\cls_j} M_z^*)^{-1}
M_{\kappa_i}|_{\clw_j},
\]
for all $w \in \D$, $j = 1, 2$, and $i = 1, \ldots, n$. Hence by
Theorem \ref{thm-2} there exists an inner function $\Theta_j \in
H^\infty_{\clb(\clw_{j}, \cle_n)}(\D)$ such that
\[
\cls_j = \Theta_j H^2_{\clw_{j}}(\D),
\]
and
\begin{equation}\label{eqn-ijtheta}
\kappa_i \Theta_j = \Theta_j \Phi_{j,i},
\end{equation}
for all $j = 1, 2$, and $i = 1, \ldots, n$. Now, let
\[
\cls_1 \subseteq \cls_2,
\]
that is
\[
\Theta_1 H^2_{\clw_{1}}(\D) \subseteq \Theta_2 H^2_{\clw_{2}}(\D).
\]
Then there exists an inner multiplier $\Psi \in
H^\infty_{\clb(\clw_{1}, \clw_{2})}(\D)$  \cite{FF} such that
\[
\Theta_1 = \Theta_2 \Psi.
\]
Using this in (\ref{eqn-ijtheta}), we get
\[
\begin{split}
\Theta_2 \Psi \Phi_{1, i} & = \Theta_1 \Phi_{1, i}
\\
& = \kappa_i \Theta_1
\\
& = \kappa_i \Theta_2 \Psi
\\
& = \Theta_2 \Phi_{2, i} \Psi,
\end{split}
\]
and so
\[
\Psi \Phi_{1, i} = \Phi_{2, i} \Psi,
\]
for all $i = 1, \ldots, n$. On the other hand, given two invariant
subspaces $\cls_j = \Theta_j H^2_{\clw_j}(\D)$, $j = 1, 2,$ for
$(M_z, M_{\kappa_1}, \ldots, M_{\kappa_n})$ on $H^2_{\cle_n}(\D)$
described as above, if there exists an inner multiplier $\Psi \in
H^\infty_{\clb(\clw_{1}, \clw_{2})}(\D)$ such that $\Theta_1 =
\Theta_2 \Psi$, then it readily follows that $\cls_1 \subseteq
\cls_2$. We state this in the following theorem:

\begin{thm}
Let $\cle$ be a Hilbert space, and let $\cls_1 = \Theta_1
H^2_{\clw_{1}}(\D)$ and $\cls_2 = \Theta_2 H^2_{\clw_{2}}(\D)$ be
two invariant subspaces for $(M_z, M_{\kappa_1}, \ldots,
M_{\kappa_n})$ on $H^2_{\cle_n}(\D)$. Let
\[
\Phi_{j,i}(w) = P_{\clw_j}(I_{\cls_j} - w P_{\cls_j} M_z^*)^{-1}
M_{\kappa_i}|_{\clw_j},
\]
for all $w \in \D$, $j = 1, 2$, and $i = 1, \ldots, n$. Then $\cls_1
\subseteq \cls_2$ if and only if there exists an inner multiplier
$\Psi \in H^\infty_{\clb(\clw_{1}, \clw_{2})}(\D)$ such that
$\Theta_1 = \Theta_2 \Psi$ and $\Psi \Phi_{1, i} = \Phi_{2, i} \Psi$
for all $i = 1, \ldots, n$.
\end{thm}

We now proceed to prove the uniqueness of the representations of
invariant subspaces as described in Theorem \ref{thm-2}. Let $\cle$
be a Hilbert space, and let $\cls$ be an invariant subspace for
$(M_z, M_{\kappa_1}, \ldots, M_{\kappa_n})$ on $H^2_{\cle_n}(\D)$.
Let $\cls = \Theta H^2_{\clw}(\D)$ and
\[
\kappa_i \Theta = \Theta {\Phi_i} \quad \quad (i = 1, \ldots, n),
\]
in the notation of Theorem \ref{thm-2}. Now assume that
$\tilde{\Theta} \in H^\infty_{\clb(\tilde{\clw})}(\D)$ is an inner
function, for some Hilbert space $\tilde{\clw}$, and
\[
\cls = \tilde{\Theta} H^2_{\tilde{\clw}}(\D).
\]
Also assume that
\[
\kappa_i \tilde{\Theta} = \tilde{\Theta} \tilde{\Phi}_i,
\]
for some shift $M_{\tilde{\Phi}_i}$ on $H^2_{\tilde{\clw}}(\D)$ and
$i = 1, \ldots, n$. Then as an application of the uniqueness of the
Beurling, Lax and Halmos inner functions (cf. Foias and Frazho,
Theorem 2.1 in page 239 \cite{FF}) to
\[
\Theta H^2_{\clw}(\D) = \tilde{\Theta} H^2_{\tilde{\clw}}(\D),
\]
we get
\[
\Theta  = \tilde{\Theta} \tau,
\]
for some unitary operator (constant in $z$) $\tau : \clw \raro
\tilde{\clw}$. Then, the previous line of argument shows that
\[
\tau \Phi_i = \tilde{\Phi}_i \tau,
\]
for all $i = 1, \ldots, n$. This proves the uniqueness of the
representations of invariant subspaces in Theorem \ref{thm-2}.

\begin{thm}\label{thm-unique}
In the setting of Theorem \ref{thm-2}, if $\cls = \tilde{\Theta}
H^2_{\tilde{\clw}}(\D)$ and $\kappa_i \tilde{\Theta} =
\tilde{\Theta} \tilde{\Phi}_i$ for some Hilbert space
$\tilde{\clw}$, inner function $\tilde{\Theta} \in
H^\infty_{\clb(\tilde{\clw})}(\D)$ and shift $M_{\tilde{\Phi}_i}$ on
$H^2_{\tilde{\clw}}(\D)$, $i = 1, \ldots, n$, then there exists a
unitary operator (constant in $z$) $\tau : \clw \raro \tilde{\clw}$
such that
\[
\Theta  = \tilde{\Theta} \tau,
\]
and
\[
\tau \Phi_i = \tilde{\Phi}_i \tau,
\]
for all $i = 1, \ldots, n$.
\end{thm}

\newsection{Applications}

In this section, first, we explore a natural connection between the
intertwining maps on vector-valued Hardy space over $\D$ and the
commutators of the multiplication operators on the Hardy space over
$\D^{n+1}$. Then, as a noteworthy added benefit to our approach, we
compute a complete set of unitary invariants for invariant subspaces
of vector-valued Hardy space over $\D^{n+1}$. We also test our main
results on invariant subspaces unitarily equivalent to
$H^2_{\cle_n}(\D)$. As a by-product, we obtain some useful results
about the structure of invariant subspaces for the Hardy space. We
begin with the following definition.

Let $\cle$ and $\tilde{\cle}$ be two Hilbert spaces. Let $\cls$ and
$\tilde{\cls}$ be invariant subspaces for the $(n+1)$-tuples of
multiplication operators on $H^2_{\cle_n}(\D)$ and
$H^2_{\tilde{\cle}_n}(\D)$, respectively. We say that $\cls$ and
$\tilde{\cls}$ are \textit{unitarily equivalent}, and write $\cls
\cong \tilde{\cls}$, if there is a unitary map $U : \cls \raro
\tilde{\cls}$ such that
\[
U M_z|_{\cls} = M_z|_{\tilde{\cls}} U \quad \text{and} \quad U
M_{\kappa_i}|_{\cls} = M_{\kappa_i}|_{\tilde{\cls}} U,
\]
for all $i = 1, \ldots, n$.

\subsection{Intertwining maps}

Recall that, given a Hilbert space $\cle$, there exists a unitary
operator $U_{\cle} : H^2_{\cle}(\D^{n+1}) \raro H^2_{\cle_n}(\D)$
(see Section 3) such that
\[
U_{\cle} M_{z_1} = M_z U_{\cle},
\]
and
\[
U_{\cle} M_{z_{i+1}} = M_{\kappa_i} U_{\cle},
\]
for all $i = 1, \ldots, n$. Let $\clf$ be another Hilbert space, and
let $X : H^2_{\cle}(\D^{n+1}) \raro H^2_{\clf}(\D^{n+1})$ be a
bounded linear operator such that
\begin{equation}\label{eq-XMz}
X M_{z_i} = M_{z_i} X,
\end{equation}
for all $i = 1, \ldots, n+1$. Set
\[
X_n = U_{\clf} X U^*_{\cle}.
\]
Then $X_n : H^2_{\cle_n}(\D) \raro H^2_{\clf_n}(\D)$ is bounded and
\begin{equation}\label{eq-XMKappa}
X_n M_z = M_z X_n \quad \text{and} \quad  X_n M_{\kappa_i} =
M_{\kappa_i} X_n,
\end{equation}
for all $i = 1, \ldots, n$. Conversely, a bounded linear operator
$X_n : H^2_{\cle_n}(\D) \raro H^2_{\clf_n}(\D)$ satisfying
\eqref{eq-XMKappa} yields a canonical bounded linear map $X :
H^2_{\cle}(\D^{n+1}) \raro H^2_{\clf}(\D^{n+1})$, namely
\[
X = U_{\clf}^* X_n U_{\cle}
\]
such that \eqref{eq-XMz} holds. Moreover, this construction shows
that $X \in \clb(H^2_{\cle}(\D^{n+1}), H^2_{\clf}(\D^{n+1}))$ is a
contraction (respectively, isometry, unitary, etc.) if and only if
$X_n \in \clb(H^2_{\cle_n}(\D), H^2_{\clf_n}(\D))$ is a contraction
(respectively, isometry, unitary, etc.).

For brevity, any map satisfying \eqref{eq-XMKappa} will be referred
to \textit{module maps}.

\subsection{A complete set of unitary invariants}

Let $\cle$ and $\tilde{\cle}$ be Hilbert spaces, and let $\{\Psi_1,
\ldots, \Psi_n\} \subseteq H^\infty_{\clb(\cle)}(\D)$ and
$\{\tilde{\Psi}_1, \ldots, \tilde{\Psi}_n\} \subseteq
H^\infty_{\clb(\tilde{\cle})}(\D)$. We say that $\{\Psi_1, \ldots,
\Psi_n\}$ and $\{\tilde{\Psi}_1, \ldots, \tilde{\Psi}_n\}$ coincide
if there exists a unitary operator $\tau : \cle \raro \tilde{\cle}$
such that
\[
\tau \Psi_i(z) = \tilde{\Psi}_i(z) \tau,
\]
for all $z \in \D$ and $i = 1, \ldots, n$.

Now let $\cls \subseteq H^2_{\cle_n}(\D)$ and $\tilde{\cls}
\subseteq H^2_{\tilde{\cle_n}}(\D)$ be invariant subspaces for
$(M_z, M_{\kappa_1}, \ldots, M_{\kappa_n})$ on $H^2_{\cle_n}(\D)$
and $H^2_{\tilde{\cle}_n}(\D)$, respectively. Let $\cls \cong
\tilde{\cls}$. By Theorem \ref{th-unit equiv}, this implies that
$(M_z, M_{\Phi_1}, \ldots, M_{\Phi_n})$ on $H^2_{\clw}(\D)$ and
$(M_z, M_{\tilde{\Phi}_1}, \ldots, M_{\tilde{\Phi}_n})$ on
$H^2_{\tilde{\clw}}(\D)$ are unitarily equivalent, where $\clw =
\cls \ominus z \cls$, $\tilde{\clw} = \tilde{\cls} \ominus z
\tilde{\cls}$ and
\[
\Phi_i(w) = P_{\clw}(I_{\cls} - w P_{\cls} M_z^*)^{-1}
M_{\kappa_i}|_{\clw},
\]
and
\[
\tilde{\Phi}_i(w) = P_{\tilde{\clw}}(I_{\tilde{\cls}} - w
P_{\tilde{\cls}} M_z^*)^{-1} M_{\kappa_i}|_{\tilde{\clw}},
\]
for all $w \in \D$ and $i = 1, \ldots, n$. Let $U : H^2_{\clw}(\D)
\raro H^2_{\tilde{\clw}}(\D)$ be a unitary map such that
\[
U M_z = M_z U,
\]
and
\[
U M_{\Phi_i} = M_{\tilde{\Phi}_i} U,
\]
for all $i = 1, \ldots, n$. The former condition implies that
\[
U = I_{H^2(\D)} \otimes \tau,
\]
for some unitary operator $\tau : \clw \raro \tilde{\clw}$, and so
the latter condition implies that
\[
\tau \Phi_i(z) = \tilde{\Phi}_i(z) \tau,
\]
for all $z \in \D$ and $i = 1, \ldots, n$. Therefore $\{\Phi_1,
\ldots, \Phi_n\}$ and $\{\tilde{\Phi}_1, \ldots, \tilde{\Phi}_n\}$
coincide. To prove the converse, assume now that the above equality
holds for a given unitary operator $\tau : \clw \raro \tilde{\clw}$.
Obviously $U = I_{H^2(\D)} \otimes \tau$ is a unitary from
$H^2_{\clw}(\D)$ to $H^2_{\tilde{\clw}}(\D)$. Clearly $U M_z = M_z
U$ and $U M_{\Phi_i} = M_{\tilde{\Phi}_i} U$ for all $i = 1, \ldots,
n$. So we have the following theorem on a complete set of unitary
invariants for invariant subspaces:

\begin{thm}\label{th-completeset}
Let $\cle$ and $\tilde{\cle}$ be Hilbert spaces. Let $\cls \subseteq
H^2_{\cle_n}(\D)$ and $\tilde{\cls} \subseteq
H^2_{\tilde{\cle_n}}(\D)$ be invariant subspaces for $(M_z,
M_{\kappa_1}, \ldots, M_{\kappa_n})$ on $H^2_{\cle_n}(\D)$ and
$H^2_{\tilde{\cle}_n}(\D)$, respectively. Then $\cls \cong
\tilde{\cls}$ if and only if $\{\Phi_1, \ldots, \Phi_n\}$ and
$\{\tilde{\Phi}_1, \ldots, \tilde{\Phi}_n\}$ coincide.
\end{thm}

Now, if we consider the Beurling, Lax and Halmos representations of
the given invariant subspaces $\cls$ and $\tilde{\cls}$ as
\[
\cls = \Theta H^2_{\clw}(\D),
\]
and
\[
\tilde{\cls} = \tilde{\Theta} H^2_{\tilde{\clw}}(\D),
\]
where $\Theta \in H^\infty_{\clb(\clw, \cle_n)}(\D)$ and
$\tilde{\Theta} \in H^\infty_{\clb(\tilde{\clw},
\tilde{\cle}_n)}(\D)$, then, in view of Remark \ref{rem-thetaw}, the
multipliers in Theorem \ref{th-completeset} can be represented as
\[
\Phi_i(w) = P_{\clw} M_{\Theta} (I_{H^2_{\clw}(\D)} - w M_z^*)^{-1}
M_{\Theta}^* M_{\kappa_i}|_{\clw},
\]
and
\[
\tilde{\Phi}_i(w) = P_{\tilde{\clw}} M_{\tilde{\Theta}}
(I_{H^2_{\tilde{\clw}}(\D)} - w M_z^*)^{-1} M_{\tilde{\Theta}}^*
M_{\kappa_i}|_{\tilde{\clw}},
\]
for all $w \in \D$ and $i = 1, \ldots, n$.

\subsection{Unitarily equivalent invariant subspaces}

Let $\cle$ and $\clf$ be Hilbert spaces, and let $X_n :
H^2_{\cle_n}(\D) \raro H^2_{\clf_n}(\D)$ be a module map. If $X_n$
is an isometry, then the closed subspace $\cls \subseteq
H^2_{\clf_n}(\D)$ defined by
\[
\cls = X_n (H^2_{\cle_n}(\D)),
\]
is invariant for $(M_z, M_{\kappa_1}, \ldots, M_{\kappa_n})$ on
$H^2_{\clf_n}(\D)$ and $\cls \cong H^2_{\cle_n}(\D)$. In other
words, the tuples $(M_z|_{\cls}, M_{\kappa_1}|_{\cls}, \ldots,
M_{\kappa_n}|_{\cls})$ on $\cls$ and $(M_z, M_{\kappa_1},\ldots,
M_{\kappa_n})$ on $H^2_{\cle_n}(\D)$ are unitarily equivalent.
Conversely, let $\cls \subseteq H^2_{\clf_n}(\D)$ be a closed
invariant subspace for $(M_z, M_{\kappa_1}, \ldots, M_{\kappa_n})$
on $H^2_{\clf_n}(\D)$, and let $\cls \cong H^2_{\cle_n}(\D)$ for
some Hilbert space $\cle$. Let $\tilde{X}_n : H^2_{\cle_n}(\D) \raro
\cls$ be the unitary map which intertwines $(M_z,
M_{\kappa_1},\ldots, M_{\kappa_n})$ on $H^2_{\cle_n}(\D)$ and
$(M_z|_{\cls}, M_{\kappa_1}|_{\cls}, \ldots, M_{\kappa_n}|_{\cls})$
on $\cls$. Suppose that $i_{\cls} : \cls \hookrightarrow
H^2_{\clf_n}(\D)$ is the inclusion map. Then
\[
X_n = i_{\cls} \circ \tilde{X}_n,
\]
is an isometry from $H^2_{\cle_n}(\D)$ to $H^2_{\clf_n}(\D)$, $X_n
M_z  = M_z X_n$, $X_n M_{\kappa_i} = M_{\kappa_i} X_n$ for all $i =
1, \ldots, n$, and
\[
\text{ran~} X_n = \cls.
\]
Therefore, if $\cls \subseteq H^2_{\clf_n}(\D)$ is a closed
invariant subspace for $(M_z, M_{\kappa_1}, \ldots, M_{\kappa_n})$
on $H^2_{\clf_n}(\D)$, then $\cls \cong H^2_{\cle_n}(\D)$, for some
Hilbert space $\cle$, if and only if there exists an isometric
module map $X_n : H^2_{\cle_n}(\D) \raro H^2_{\clf_n}(\D)$ such that
$\cls = X_n (H^2_{\cle_n}(\D))$. Now, it also follows from the
discussion at the beginning of this section that $X :
H^2_{\cle}(\D^{n+1}) \raro H^2_{\clf}(\D^{n+1})$ (corresponding to
the module map $X_n$) is an isometry and $X M_{z_i} = M_{z_i} X$ for
all $i = 1, \ldots, n$. Then Theorem \ref{thm-fibre} tells us that
\[
\text{dim~} \cle \leq \text{dim~} \clf.
\]
Therefore, we have the following theorem:
\begin{thm}\label{th-ues}
Let $\cle$ and $\clf$ be Hilbert spaces, and let  $\cls \subseteq
H^2_{\clf_n}(\D)$ be a closed invariant subspace for $(M_z,
M_{\kappa_1},\ldots, M_{\kappa_n})$ on $H^2_{\clf_n}(\D)$. Then
$\cls \cong H^2_{\cle_n}(\D)$ if and only if there exists an
isometric module map $X_n : H^2_{\cle_n}(\D) \raro H^2_{\clf_n}(\D)$
such that
\[
\cls = X_n H^2_{\cle_n}(\D).
\]
Moreover, in this case
\[
\text{dim~} \cle \leq \text{dim~} \clf.
\]
\end{thm}

Of particular interest is the case when $\clf = \mathbb{C}$. In this
case (see Section 3) the tensor product Hilbert space $\clf_n =
H^2(\D^n) \otimes \mathbb{C}$ is denoted by $H_n$, that is, $H_n =
H^2(\D^n)$.

\begin{cor}\label{co-ues}
Let $\cls \subseteq H^2_{H_n}(\D)$ be a closed invariant subspace
for $(M_z, M_{\kappa_1},\ldots, M_{\kappa_n})$ on $H^2_{H_n}(\D)$.
Then $\cls \cong H^2_{H_n}(\D)$ if and only if there exists an
isometric module map $X_n : H^2_{H_n}(\D) \raro H^2_{H_n}(\D)$ such
that
\[
\cls = X_n (H^2_{H_n}(\D)).
\]
\end{cor}

The above result, in the polydisc setting, was first observed by
Agrawal, Clark and Douglas (see Corollary 1 in \cite{ACD}). Also see
Mandrekar \cite{M}.

We now proceed to analyze doubly commuting invariant subspaces. Let
$\clf$ be a Hilbert space, and let $\cls \subseteq H^2_{\clf_n}(\D)$
be a closed invariant subspace for $(M_z, M_{\kappa_1},\ldots,
M_{\kappa_n})$ on $H^2_{\clf_n}(\D)$. Set
\[
V = M_z|_{\cls},
\]
and
\[
V_i = M_{\kappa_i}|_{\cls},
\]
for all $i = 1, \ldots, n$. We say that $\cls$ is \textit{doubly
commuting} if $V_i^* V_j = V_j V_i^*$ for all $1 \leq i < j \leq n$,
and $V V_l^* = V_l^* V$ for all $l = 1, \ldots, n$.

Now let $\cle$ be a Hilbert space, and suppose that
$H^2_{\cle_n}(\D) \cong \cls$. In view of Theorem \ref{th-ues} this
implies that $(V, V_1, \ldots, V_n)$ on $\cls$ and $(M_z,
M_{\kappa_1},\ldots, M_{\kappa_n})$ on $H^2_{\cle_n}(\D)$ are
unitarily equivalent. Because $H^2_{\cle_n}(\D)$ is doubly commuting
this immediately implies that $\cls$ is doubly commuting.

Conversely, let $\cls$ is doubly commuting. From Theorem
\ref{th-unit equiv} we readily conclude that $(M_z, M_{\Phi_1},
\ldots, M_{\Phi_n})$ on $H^2_{\clw}(\D)$ and $(V, V_1, \ldots, V_n)$
on $\cls$ are unitarily equivalent.

Applying Theorem \ref{thm-dcequiv} with $(M_z, M_{\Phi_1}, \ldots,
M_{\Phi_n})$ in place of $(M_z, M_{\Psi_1}, \ldots, M_{\Psi_n})$, we
see that $(V, V_1, \ldots, V_n)$ on $\cls$  and $(M_z,
M_{\kappa_1},\ldots, M_{\kappa_n})$ on $H^2_{\cle_n}(\D)$ are
unitarily equivalent, where $\cle$ is a Hilbert space. Now,
proceeding as in the proof of the necessary part of Theorem
\ref{th-ues} one checks that there exists a module isometry $X_n :
H^2_{\cle_n}(\D) \raro H^2_{\clf_n}(\D)$ such that
\[
\text{ran~} X_n = \cls.
\]
This proves the following variant of Theorem \ref{th-ues}:

\begin{thm}
Let $\clf$ be a Hilbert space. An invariant subspace $\cls \subseteq
H^2_{\clf_n}(\D)$ is doubly commuting if and only if there exists a
Hilbert space $\cle$ and an isometric module map $X_n :
H^2_{\cle_n}(\D) \raro H^2_{\clf_n}(\D)$ such that
\[
\cls = X_n H^2_{\cle_n}(\D).
\]
Moreover, in this case
\[
\text{dim~} \cle \leq \text{dim~} \clf.
\]
\end{thm}

The above result, in the polydisc setting, was first observed by
Mandrekar \cite{M}. Also this should be compared with the discussion
prior to Corollary \ref{cor-ueqv} on the application of the
classical Beurling, Lax and Halmos theorem to invariant subspaces of
the Hardy space over the unit disc.

\newsection{Appendix: An inequality on fibre dimensions}\label{sec-app}

Given a Hilbert space $\cle$, the $n$-tuple of multiplication
operators by the coordinate functions $z_i$, $i = 1, \ldots, n$, on
$H^2_{\cle}(\D^n)$ is denoted by $(M_{z_1}^{\cle}, \ldots,
M_{z_n}^{\cle})$. Whenever $\cle$ is clear from the context, we will
omit the superscript $\cle$. Clearly, one can regard $\cle$ as a
closed subspace of $H^2_{\cle}(\D^n)$ by identifying $\cle$ with the
constant $\cle$-valued functions on $\D^n$.

In this appendix, we aim to prove the following result:

\begin{thm}\label{thm-fibre}
Let $\cle_1$ and $\cle_2$ be Hilbert spaces and let $X :
H^2_{\mathcal{E}_1}(\mathbb{D}^n) \rightarrow
H^2_{\mathcal{E}_2}(\mathbb{D}^n)$ be an isometry. If
\[
X M_{z_i}^{\cle_1} = M_{z_i}^{\cle_2} X,
\]
for all $i = 1, \ldots, n$, then
\[
\dim \mathcal{E}_1 \leq \dim \mathcal{E}_2.
\]
\end{thm}

We believe that the above result (possibly) follows from the
boundary behavior of bounded analytic functions following the
classical case $n = 1$ (see the remark at the end of this appendix).
Here, however, we take a shorter approach than generalizing the
classical theory of bounded analytic functions on the unit polydisc.
We first prove the $L^2$-version of the above statement.

\begin{thm}\label{thm-Lfibre}
Let $\cle_1$ and $\cle_2$ be Hilbert spaces and let $\tilde{X} :
L^2_{\mathcal{E}_1}(\mathbb{T}^n) \rightarrow
L^2_{\mathcal{E}_2}(\mathbb{T}^n)$ be an isometry. If
\[
\tilde{X} M_{e^{i\theta_j}} = M_{e^{i\theta_j}} \tilde{X},
\]
for all $j = 1, \ldots, n$, then
\[
\dim \mathcal{E}_1 \leq \dim \mathcal{E}_2.
\]
\end{thm}

\begin{proof}
By the triviality, we can assume that
\[
m : = \dim \cle_2 < \infty.
\]
Let $\{\eta_j\}_{j=1}^m$ be an orthonormal basis for
$\mathcal{E}_2$. Since $\{e_{\bm{k}} :\bm{k} \in \mathbb{Z}^n\}$,
where
\[
e_{\bm{k}} = \prod_{j=1}^n e^{i k_j \theta_j} \quad \quad (\bm{k}
\in \mathbb{Z}^n),
\]
is an orthonormal basis for $L^2(\mathbb{T}^n)$, this implies that
$\{e_{\bm{k}} \eta_j :\bm{k} \in \mathbb{Z}^n, j = 1, \ldots, n\}$
is an orthonormal basis for $L^2_{\mathcal{E}_2}(\mathbb{T}^n)$. Let
$\{f_j: j \in J\}$ be an orthonormal basis for
$\tilde{X}(\mathcal{E}_1)$, where $J$ is a subset of $\mathbb{Z}_+$.
In view of the intertwining property of $\tilde{X}$, this implies
that $\{e_{\bm{k}} f_j: \bm{k} \in \mathbb{Z}^n, j \in J\}$ is an
orthonormal basis for
\[
\widetilde{X}( L^2_{\mathcal{E}_1}(\mathbb{T}^n)) \subseteq
L^2_{\mathcal{E}_2}(\mathbb{T}^n),
\]
and so, an orthonormal set in $L^2_{\mathcal{E}_2}(\mathbb{T}^n)$.
It follows from the Parseval's identity that
\[
\begin{split}
\dim \cle_1 & = \dim (\tilde{X} \cle_1)
\\
& = \sum_{j \in J} \|f_j\|^2
\\
& = \sum_{j \in J} \sum_{l=1}^m \sum_{\bm{k} \in \mathbb{Z}^n}|
\langle M^{\bm{k}}_{e^{i \theta}}\eta_l, f_j \rangle |^2
\\
& = \sum_{j \in J} \sum_{l=1}^m \sum_{\bm{k} \in \mathbb{Z}^n}|
\langle \eta_l, M^{\bm{k}}_{e^{i \theta}} f_j \rangle |^2
\\
& = \sum_{j \in J} \sum_{l=1}^m \sum_{\bm{k} \in \mathbb{Z}^n}|
\langle \eta_l, e_{\bm{k}} f_j \rangle |^2,
\end{split}
\]
on the one hand, and on the other, by Bessel's Inequality,
\[
\begin{split}
m & = \sum_{l=1}^m \|\eta_l\|^2
\\
& \geq \sum_{l=1}^m \sum_{j \in J} \sum_{\bm{k} \in \mathbb{Z}^n}|
\langle \eta_l, e_{\bm{k}} f_j \rangle |^2.
\end{split}
\]
This proves $\dim \mathcal{E}_1 \leq m$ and completes the proof of
the theorem.
\end{proof}

\NI\textit{Proof of Theorem \ref{thm-fibre}:} Define $\tilde{X}$ on
$\{e_{\bm{k}} \eta: \bm{k} \in \mathbb{Z}^n, \eta \in \cle_1\}$ by
\[
\tilde{X} (e_{\bm{k}} \eta) = e_{\bm{k}} X \eta,
\]
for all $\bm{k} \in \mathbb{Z}^n$ and $\eta \in \cle_1$. The
intertwining property of the isometry $X$ then gives
\[
\langle \tilde{X} (e_{\bm{k}} \eta), \tilde{X} (e_{\bm{l}} \zeta)
\rangle_{L^2_{\cle_2}(\mathbb{T}^n)} = \langle e_{\bm{k}} \eta,
e_{\bm{l}} \zeta \rangle_{L^2_{\cle_1}(\mathbb{T}^n)},
\]
for all $\bm{k}, \bm{l} \in \mathbb{Z}^n$ and $\eta, \zeta \in
\cle_1$. Therefore this map extends uniquely to an isometry, denoted
again by $\tilde{X}$ from $L^2_{\cle_1}(\mathbb{T}^n)$ to
$L^2_{\cle_2}(\mathbb{T}^n)$, such that
\[
\tilde{X} M_{e^{i\theta_j}} = M_{e^{i\theta_j}} \tilde{X},
\]
for all $j = 1, \ldots, n$. The result then easily follows from
Theorem \ref{thm-Lfibre}. \qed

If $X : H^2_{\cle_1}(\D^n) \raro H^2_{\cle_2}(\D^n)$ is an isometry,
and if $X M_{z_i} = M_{z_i} X$ for all $i = 1, \ldots, n$, then it
is easy to see that
\[
X = M_{\Theta},
\]
for some isometric multiplier $\Theta \in H^\infty_{\clb(\cle_1,
\cle_2)}(\D^n)$ (that is, $M_{\Theta} : H^2_{\cle_1}(\D^n) \raro
H^2_{\cle_2}(\D^n)$ is an isometry). In the case $n = 1$, the
conclusion of Theorem \ref{thm-fibre} follows from the boundary
behavior of bounded analytic functions on the open unit disc:
$M_{\Theta}$ is an isometry if and only if $\Theta(e^{i\theta})$ is
isometry a.e. on $\mathbb{T}$ (cf. \cite{NF}). Unlike the proof of
the classical case $n = 1$, our proof does not use the boundary
behavior of $\Theta$.

\vspace{0.3in}

\NI\textit{Acknowledgement:} The first author's research work is
supported by National Post Doctoral Fellowship (N-PDF), File No.
PDF/2017/001856. The research of the third author was supported in
part by NBHM (National Board of Higher Mathematics, India) Research
Grant NBHM/R.P.64/2014.


\begin{thebibliography}{99}

\vspace{1cm}

\bibitem{AM}
J. Agler and J. McCarthy, \emph{Distinguished varieties}, Acta.
Math. 194 (2005), 133–-153.

\bibitem{ACD}
O. Agrawal, D. Clark, and R. Douglas, {\em Invariant subspaces in
the polydisk}, Pacific J. Math. 121 (1986), 1-–11.

\bibitem{AC}
P. Ahern and D. Clark, {\em Invariant subspaces and analytic
continuation in several variables}, J. Math. Mech. 19 (1970),
963–-969.


\bibitem{AA}
A. Aleman, {\em Invariant subspaces with finite codimension in
Bergman spaces}, Trans. Amer. Math. Soc. 330 (1992), 531-–544.

\bibitem{AB}
S. Axler and P. Bourdon, {\em Finite-codimensional invariant
subspaces of Bergman spaces}, Trans. Amer. Math. Soc. 306 (1988),
805–-817.

\bibitem{BCL}
C. A. Berger, L. A. Coburn and A. Lebow, {\em Representation and
index theory for $C^*$-algebras generated by commuting isometries},
J. Funct. Anal. 27 (1978), no. 1, 51–-99.

\bibitem{B}
A. Beurling, {\em On two problems concerning linear transformations
in Hilbert space}, Acta Math., 81 (1949), 239--255.

\bibitem{CG}
X. Chen and K. Guo, {\em Analytic Hilbert modules}, Chapman \&
Hall/CRC Research Notes in Mathematics, 433. Chapman \& Hall/CRC,
Boca Raton, FL, 2003.


\bibitem{DPSY}
R. Douglas, V. Paulsen, C.-H. Sah and K. Yan, {\em Algebraic
reduction and rigidity for Hilbert modules}, Amer. J. Math. 117
(1995), 75–-92.


\bibitem{DY}
R. Douglas and K. Yan, {\em On the rigidity of Hardy submodules},
Integral Equations Operator Theory 13 (1990), no. 3, 350–363.


\bibitem{F1}
X. Fang, {\em Additive invariants on the Hardy space over the
polydisc}, J. Funct. Anal. 253 (2007), 359–-372.

\bibitem{FF}
C. Foias and A. E. Frazho, {\em The commutant lifting approach to
interpolation problems}, Operator Theory: Advances and Applications,
44. Birkhuser Verlag, Basel, 1990.

\bibitem{G}
T. W. Gamelin, {\em Embedding Riemann surfaces in maximal ideal
spaces}, J. Functional Analysis. (1968) 123-–146.

\bibitem{KG1}
K. Guo, {\em Defect operators for submodules of $H^2_d$}, J. Reine
Angew. Math. 573 (2004), 181-–209.

\bibitem{KG}
K. Guo, {\em Algebraic reduction for Hardy submodules over polydisk
algebras}, J. Operator Theory 41 (1999), 127-–138.

\bibitem{GY}
K. Guo and R. Yang, {\em The core function of submodules over the
bidisk}, Indiana Univ. Math. J. 53 (2004), 205–-222.


\bibitem{H}
P. Halmos, {\em Shifts on Hilbert spaces}, J. Reine Angew. Math. 208
(1961) 102--112.

\bibitem{I}
K. Izuchi, {\em Unitary equivalence of invariant subspaces in the
polydisk}, Pacific Journal of Mathematics 130 (1987), 351–-358.

\bibitem{L}
P. Lax, {\em Translation invariant spaces}, Acta Math. 101 (1959)
163--178.

\bibitem{MSS}
A. Maji, J. Sarkar and Sankar T. R, {\em Pairs of Commuting
Isometries - I}, arXiv:1708.02609

\bibitem{M}
V. Mandrekar, {\em The validity of Beurling theorems in polydiscs},
Proc. Amer. Math. Soc. 103 (1988), 145–-148.


\bibitem{NF}
B. Sz.-Nagy and C. Foias, {\em Harmonic analysis of operators on
Hilbert space.} North-Holland, Amsterdam-London, 1970.

\bibitem{P}
M. Putinar, On invariant subspaces of several variable Bergman
spaces, Pacific J. Math. 147 (1991), 355–-364.

\bibitem{JS}
J. Sarkar, {\em Wold decomposition for doubly commuting isometries},
Linear Algebra Appl. 445 (2014), 289–-301.

\bibitem{Ru}
W. Rudin, {\em Function Theory in Polydiscs}, Benjamin, New York,
1969.

\bibitem{WR}
W. Rudin, {\em Invariant subspaces of $H^2$ on a torus}, J. Funct.
Anal. 61 (1985), 378–-384.

\bibitem{GZ}
K. Guo, S. Sun, D. Zheng and C. Zhong, {\em Multiplication operators
on the Bergman space via the Hardy space of the bidisk}, J. Reine
Angew. Math. 628 (2009), 129-–168.

\bibitem{VN}
J. von Neumann, {\em Allgemeine Eigenwerttheorie Hermitescher
Funktionaloperatoren}, Math. Ann., 102 (1929), 49–-131.

\bibitem{W}
H. Wold, {\em A study in the analysis of stationary time series},
Stockholm, 1954.

\bibitem{QY}
R. Yang, {\em Operator theory in the Hardy space over the bidisk.
III}, J. Funct. Anal. 186 (2001),  521-–545.


\end{thebibliography}
\end{document}